\documentclass[11pt,reqno]{amsart}
\usepackage{amsfonts,amssymb,amsmath,color}
\usepackage{hyperref}

\newtheorem{theorem}{Theorem}

\newtheorem{proposition}[theorem]{Proposition}

\newtheorem{remark}[theorem]{Remark}

\numberwithin{equation}{section} \numberwithin{theorem}{section}

\begin{document}

\title{Pattern Recognition on Oriented Matroids: Symmetric Cycles in the Hypercube Graphs. III}

\author{Andrey O. Matveev}
\email{andrey.o.matveev@gmail.com}

\begin{abstract}
We present statistics on the decompositions (with respect to a distinguished symmetric $2t$-cycle) of vertices of the hypercube graph,
whose negative parts are regarded as disjoint unions of two subsets of the ground set~$\{1,\ldots,t\}$ of the corresponding oriented matroid.
\end{abstract}

\maketitle

\pagestyle{myheadings}

\markboth{PATTERN RECOGNITION ON ORIENTED MATROIDS}{A.O.~MATVEEV}

\thispagestyle{empty}



\section{Introduction}

We continue to study the decompositions of vertices $T\in\{1,-1\}^t$, $t\geq 3$, of the {\em hypercube graph\/} $\boldsymbol{H}(t,2)$ with respect to its distinguished symmetric $2t$-cycle.

Recall that vertices $T':=(T'(1),\ldots,T'(t))$ and $T'':=(T''(1),\ldots,T''(t))$ are adjacent in the graph~$\boldsymbol{H}(t,2)$ if and only if $|\{e\in E_t\colon T''(e)=-T'(e)\}|=1$, where $E_t:=[t]:=\{1,\ldots,t\}$ is the {\em ground set\/} of the corresponding {\em oriented matroid\/} $\mathcal{H}:=(E_t,\{1,-1\}^t)$, with the set of {\em topes\/}~$\{1,-1\}^t$, which is realizable as the {\em arrangement\/} of {\em coordinate hyperplanes\/} in $\mathbb{R}^t$, see, e.g.,~\cite[Example~4.1.4]{BLSWZ}.

We denote by $\mathrm{T}^{(+)}:=(1,\ldots,1)$ the {\em positive tope\/} of the oriented matroid~$\mathcal{H}$; the {\em negative tope\/} $-\mathrm{T}^{(+)}$ is denoted by~$\mathrm{T}^{(-)}$. Given a subset~$A\subseteq E_t$, we let ${}_{-A}\mathrm{T}^{(+)}$ denote the tope $T$ whose {\em negative part\/} $T^-:=\{e\in E_t\colon T(e)=-1\}$ is the set~$A$, that is,
\begin{equation*}
({}_{-A}\mathrm{T}^{(+)})(e):=
\begin{cases}
-1\; , & \text{if $e\in A$}\; ,\\
\phantom{-}1\; , & \text{if $e\not\in A$}\; ;
\end{cases}
\end{equation*}
if $s\in E_t$, then we write ${}_{-s}\mathrm{T}^{(+)}$ instead of ${}_{-\{s\}}\mathrm{T}^{(+)}$.

As earlier in~\cite{M-SC-II}, throughout this note we will be dealing exclusively with one distinguished {\em symmetric $2t$-cycle\/} $\boldsymbol{R}:=(R^0,R^1,\ldots,R^{2t-1},R^0)$ in the
graph~$\boldsymbol{H}(t,2)$, whose vertex sequence is as follows:
\begin{equation}
\label{eq:12}
\begin{split}
R^0:\!&=\mathrm{T}^{(+)}\; ,\\
R^s:\!&={}_{-[s]}R^0\; ,\ \ \ 1\leq s\leq t-1\; ,
\end{split}
\end{equation}
and
\begin{equation}
\label{eq:13}
R^{k+t}:=-R^k\; ,\ \ \ 0\leq k\leq t-1\; .
\end{equation}
The subsequence of vertices $(R^0,\ldots,R^{t-1})$ is a basis of the space~$\mathbb{R}^t$; for any vertex $T$ of the graph~$\boldsymbol{H}(t,2)$, there exists a unique vector~$\boldsymbol{x}:=\boldsymbol{x}(T):=\boldsymbol{x}(T,\boldsymbol{R}):=(x_1,\ldots,x_t)\in\{-1,0,1\}^t$ such that
\begin{equation}
\label{eq:14}
T=\sum_{i\in [t]}x_i\cdot R^{i-1}\; .
\end{equation}
Thus, the set
\begin{equation*}
\boldsymbol{Q}(T,\boldsymbol{R}):=\{x_i\cdot R^{i-1}\colon x_i\neq 0\}
\end{equation*}
is the unique inclusion-minimal subset of
vertices
of the cycle~$\boldsymbol{R}$ such that
\begin{equation*}
\sum_{Q\in\boldsymbol{Q}(T,\boldsymbol{R})}Q=T\; ,
\end{equation*}
see~\cite[Section~11.1]{AM-PROM-I},\cite[\S{}1]{M-SC-II}.
In particular, we have~$|\boldsymbol{Q}(T,\boldsymbol{R})|=\|\boldsymbol{x}\|^2=:\langle\boldsymbol{x},\boldsymbol{x}\rangle$, where
$\langle\cdot,\cdot\rangle$ is the standard scalar product on~$\mathbb{R}^t$.

In Theorem~\ref{th:3} we obtain statistics on partitions $T^-=A\dot\cup B$ of the negative parts of vertices $T$ of the hypercube graph $\boldsymbol{H}(t,2)$ into two subsets and on the decompositions of vertices, $\boldsymbol{Q}({}_{-A}\mathrm{T}^{(+)},\boldsymbol{R})$, $\boldsymbol{Q}({}_{-B}\mathrm{T}^{(+)},\boldsymbol{R})$ and~$\boldsymbol{Q}(T,\boldsymbol{R})$,  with respect to the symmetric cycle~$\boldsymbol{R}$.

\section{Partitions of the negative parts of topes into two subsets, and the decompositions of topes}

Given two vertices $T'$ and $T''$ of the graph~$\boldsymbol{H}(t,2)$, we denote by $\mathbf{S}(T',T'')$
their {\em separation set}, that is, $\mathbf{S}(T',T''):=\{e\in E_t\colon T'(e)=-T''(e)\}$.
Note that if $x_e(T')\neq 0$ and $x_e(T'')\neq 0$, for some $e\in E_t$, then
\begin{align*}
e\not\in\mathbf{S}(T',T'')\ \ \ &\Longrightarrow\ \ \ x_e(T'')=\phantom{-}x_e(T')\; ;\\
e\in\mathbf{S}(T',T'')\ \ \ &\Longrightarrow\ \ \ x_e(T'')=-x_e(T')\; ,
\end{align*}
see~\cite[\S{}2]{M-SC-II}.

For elements~$s\in E_t$, we denote by $\boldsymbol{\sigma}(s):=(0,\ldots,\underset{\overset{\uparrow}{s}}{1},\ldots,0)$ the vectors of the standard basis of $\mathbb{R}^t$, and we define vectors~$\boldsymbol{y}(s):=\boldsymbol{y}(s;t)$ by
\begin{equation*}
\boldsymbol{y}(s):=\boldsymbol{x}({}_{-s}\mathrm{T}^{(+)})\; .
\end{equation*}

\begin{remark}
\label{th:1}
It follows from~{\rm\cite[Rem.~2.2]{M-SC-II}} that if $A$ and $B$ are {\em disjoint\/}
subsets of the ground set $E_t$, then
\begin{equation*}
\begin{split}
\boldsymbol{x}({}_{-(A\dot\cup B)}\mathrm{T}^{(+)})&=(1-|A|-|B|)\cdot\boldsymbol{\sigma}(1)+\sum_{s\in A}\boldsymbol{y}(s)
+\sum_{s\in B}\boldsymbol{y}(s)\\
&=\boldsymbol{x}({}_{-A}\mathrm{T}^{(+)})-|B|\cdot\boldsymbol{\sigma}(1)+\sum_{s\in B}\boldsymbol{y}(s)\\&=
-|A|\cdot\boldsymbol{\sigma}(1)+\sum_{s\in A}\boldsymbol{y}(s)+\boldsymbol{x}({}_{-B}\mathrm{T}^{(+)})\; ,
\end{split}
\end{equation*}
that is,
\begin{equation*}
\boldsymbol{x}({}_{-(A\dot\cup B)}\mathrm{T}^{(+)})=
-\boldsymbol{\sigma}(1)+\boldsymbol{x}({}_{-A}\mathrm{T}^{(+)})+\boldsymbol{x}({}_{-B}\mathrm{T}^{(+)})\; .
\end{equation*}
\end{remark}

As a consequence, we have
\begin{multline*}
\|\boldsymbol{x}({}_{-(A\dot\cup B)}\mathrm{T}^{(+)})\|^2=
\|\boldsymbol{x}({}_{-A}\mathrm{T}^{(+)})+\boldsymbol{x}({}_{-B}\mathrm{T}^{(+)})\|^2\\
-2\bigl(x_1({}_{-A}\mathrm{T}^{(+)})+x_1({}_{-B}\mathrm{T}^{(+)})\bigr)+1\\=
\|\boldsymbol{x}({}_{-A}\mathrm{T}^{(+)})\|^2
+\|\boldsymbol{x}({}_{-B}\mathrm{T}^{(+)})\|^2+2\langle\boldsymbol{x}({}_{-A}\mathrm{T}^{(+)}),\boldsymbol{x}({}_{-B}\mathrm{T}^{(+)})\rangle\\
-2x_1({}_{-A}\mathrm{T}^{(+)})-2x_1({}_{-B}\mathrm{T}^{(+)})+1\; .
\end{multline*}

\section{Inclusion of the negative parts of topes, and the decompositions of topes}

Let us consider the vectors~$\boldsymbol{x}(T')$ and~$\boldsymbol{x}(T'')$ associated with vertices $T',T''\in\{1,-1\}^t$ of $\boldsymbol{H}(t,2)$, whose negative parts are comparable by inclusion.

\begin{proposition} Let $A$ and $C$ be two sets such that $A\subseteq C\subseteq E_t$.
\begin{itemize}
\item[\rm(i)]
If
\begin{equation*}
\{1,t\}\cap A=\{1,t\}\cap C\; ,
\end{equation*}
then
\begin{equation*}
\boldsymbol{x}({}_{-C}\mathrm{T}^{(+)})=\boldsymbol{x}({}_{-A}\mathrm{T}^{(+)})-\sum_{i\in C-A}(\;\boldsymbol{\sigma}(i)-\boldsymbol{\sigma}(i+1)\;)\; .
\end{equation*}

\item[\rm(ii)]
If
\begin{equation*}
\{1,t\}\cap A=\{1\}\ \ \ \text{and}\ \ \ \{1,t\}\cap C=\{1,t\}\; ,
\end{equation*}
or
\begin{equation*}
|\{1,t\}\cap A|=0\ \ \ \text{and}\ \ \ \{1,t\}\cap C=\{t\}\; ,
\end{equation*}
then
\begin{equation*}
\boldsymbol{x}({}_{-C}\mathrm{T}^{(+)})=\boldsymbol{x}({}_{-A}\mathrm{T}^{(+)})-\boldsymbol{\sigma}(1)-\boldsymbol{\sigma}(t)-\sum_{i\in C-(A\dot\cup\{t\})}(\;\boldsymbol{\sigma}(i)-\boldsymbol{\sigma}(i+1)\;)\; .
\end{equation*}

\item[\rm(iii)]
If
\begin{equation*}
|\{1,t\}\cap A|=0\ \ \ \text{and}\ \ \ \{1,t\}\cap C=\{1\}\; ,
\end{equation*}
or
\begin{equation*}
\{1,t\}\cap A=\{t\}\ \ \ \text{and}\ \ \ \{1,t\}\cap C=\{1,t\}\; ,
\end{equation*}
then
\begin{equation*}
\boldsymbol{x}({}_{-C}\mathrm{T}^{(+)})=\boldsymbol{x}({}_{-A}\mathrm{T}^{(+)})-\boldsymbol{\sigma}(1)+\boldsymbol{\sigma}(2)-\sum_{i\in C-(A\dot\cup\{1\})}(\;\boldsymbol{\sigma}(i)-\boldsymbol{\sigma}(i+1)\;)\; .
\end{equation*}

\item[\rm(iv)]
If
\begin{equation*}
|\{1,t\}\cap A|=0\ \ \ \text{and}\ \ \ \{1,t\}\cap C=\{1,t\}\; ,
\end{equation*}
then
\begin{equation*}
\boldsymbol{x}({}_{-C}\mathrm{T}^{(+)})=\boldsymbol{x}({}_{-A}\mathrm{T}^{(+)})-2\boldsymbol{\sigma}(1)
+\boldsymbol{\sigma}(2)-\boldsymbol{\sigma}(t)-\sum_{i\in C-(A\dot\cup\{1,t\})}(\;\boldsymbol{\sigma}(i)-\boldsymbol{\sigma}(i+1)\;)\; .
\end{equation*}
\end{itemize}
\end{proposition}

\begin{proof} We will use~\cite[Rem.~2.3]{M-SC-II}:

\noindent{\rm(a)}
If
\begin{equation*}
\{1,t\}\cap C=\{1\}\; ,
\end{equation*}
then
\begin{multline*}
A\subseteq C\ \ \ \Longrightarrow \ \ \ \boldsymbol{x}({}_{-A}\mathrm{T}^{(+)})=
\begin{cases}
\phantom{-}\boldsymbol{\sigma}(2)-\sum_{i\in A-\{1\}}(\;\boldsymbol{\sigma}(i)-\boldsymbol{\sigma}(i+1)\;)\; ,
\\ \hspace{6cm}\text{if $1\in A$}\; ,\\
\phantom{-}\boldsymbol{\sigma}(1)-\sum_{i\in A}(\;\boldsymbol{\sigma}(i)-\boldsymbol{\sigma}(i+1)\;)\; ,\ \, \text{if $1\not\in A$}\;
\end{cases}\\
=
\begin{cases}
\boldsymbol{x}({}_{-C}\mathrm{T}^{(+)})+\sum_{i\in C-A}(\;\boldsymbol{\sigma}(i)-\boldsymbol{\sigma}(i+1)\;)\; , & \text{if $1\in A$}\; ,\\
\boldsymbol{x}({}_{-C}\mathrm{T}^{(+)})+\boldsymbol{\sigma}(1)-\boldsymbol{\sigma}(2)+\sum_{i\in C-(A\dot\cup\{1\})}(\;\boldsymbol{\sigma}(i)-\boldsymbol{\sigma}(i+1)\;)\; , & \text{if $1\not\in A$}\; .
\end{cases}
\end{multline*}

\noindent{\rm(b)}
If
\begin{equation*}
\{1,t\}\cap C=\{1,t\}\; ,
\end{equation*}
then
\begin{multline*}
A\subseteq C\ \ \ \Longrightarrow \ \ \ \\ \boldsymbol{x}({}_{-A}\mathrm{T}^{(+)})=
\begin{cases}
\phantom{-}\boldsymbol{\sigma}(2)-\sum_{i\in A-\{1\}}(\;\boldsymbol{\sigma}(i)-\boldsymbol{\sigma}(i+1)\;)\; ,
\\ \hspace{6cm}\text{if $\{1,t\}\cap A=\{1\}$}\; ,\\
-\boldsymbol{\sigma}(1)+\boldsymbol{\sigma}(2)-\boldsymbol{\sigma}(t)-\sum_{i\in A-\{1,t\}}(\;\boldsymbol{\sigma}(i)-\boldsymbol{\sigma}(i+1)\;)\; , \\ \hspace{6cm}\text{if $\{1,t\}\cap A=\{1,t\}$}\; ,\\
\phantom{-}\boldsymbol{\sigma}(1)-\sum_{i\in A}(\;\boldsymbol{\sigma}(i)-\boldsymbol{\sigma}(i+1)\;)\; ,\ \text{if $|\{1,t\}\cap A|=0$}\; ,\\
-\boldsymbol{\sigma}(t)-\sum_{i\in A-\{t\}}(\;\boldsymbol{\sigma}(i)-\boldsymbol{\sigma}(i+1)\;)\; ,
\\ \hspace{6cm}\text{if $\{1,t\}\cap A=\{t\}$}\;
\end{cases}
\\=
\begin{cases}
\boldsymbol{x}({}_{-C}\mathrm{T}^{(+)})+\boldsymbol{\sigma}(1)+\boldsymbol{\sigma}(t)+\sum_{i\in C-(A\dot\cup\{t\})}(\;\boldsymbol{\sigma}(i)-\boldsymbol{\sigma}(i+1)\;)\; ,
\\ \hspace{8cm}\text{if $\{1,t\}\cap A=\{1\}$}\; ,\\
\boldsymbol{x}({}_{-C}\mathrm{T}^{(+)})+\sum_{i\in C-A}(\;\boldsymbol{\sigma}(i)-\boldsymbol{\sigma}(i+1)\;)\; ,
\ \ \ \ \ \ \ \ \!\text{if $\{1,t\}\cap A=\{1,t\}$}\; ,\\
\boldsymbol{x}({}_{-C}\mathrm{T}^{(+)})+2\boldsymbol{\sigma}(1)-\boldsymbol{\sigma}(2)+\boldsymbol{\sigma}(t)+\sum_{i\in C-(A\dot\cup\{1,t\})}(\;\boldsymbol{\sigma}(i)-\boldsymbol{\sigma}(i+1)\;)\; ,
\\ \hspace{8cm}\text{if $|\{1,t\}\cap A|=0$}\; ,\\
\boldsymbol{x}({}_{-C}\mathrm{T}^{(+)})+\boldsymbol{\sigma}(1)-\boldsymbol{\sigma}(2)+\sum_{i\in C-(A\dot\cup\{1\})}(\;\boldsymbol{\sigma}(i)-\boldsymbol{\sigma}(i+1)\;)\; ,
\\ \hspace{8cm}\text{if $\{1,t\}\cap A=\{t\}$}\; .
\end{cases}
\end{multline*}

\noindent{\rm(c)}
If
\begin{equation*}
|\{1,t\}\cap C|=0\; ,
\end{equation*}
then
\begin{multline*}
A\subseteq C\ \ \ \Longrightarrow \ \ \
\boldsymbol{x}({}_{-A}\mathrm{T}^{(+)})=
\boldsymbol{\sigma}(1)-\sum_{i\in A}(\;\boldsymbol{\sigma}(i)-\boldsymbol{\sigma}(i+1)\;)\\=
\boldsymbol{x}({}_{-C}\mathrm{T}^{(+)})+\sum_{i\in C-A}(\;\boldsymbol{\sigma}(i)-\boldsymbol{\sigma}(i+1)\;)
\; .
\end{multline*}

\noindent{\rm(d)}
If
\begin{equation*}
\{1,t\}\cap C=\{t\}\; ,
\end{equation*}
then
\begin{multline*}
A\subseteq C\ \ \ \Longrightarrow \ \ \ \\ \boldsymbol{x}({}_{-A}\mathrm{T}^{(+)})=
\begin{cases}
\phantom{-}\boldsymbol{\sigma}(1)-\sum_{i\in A}(\;\boldsymbol{\sigma}(i)-\boldsymbol{\sigma}(i+1)\;)\; , & \text{if $t\not\in A$}\; ,\\
-\boldsymbol{\sigma}(t)-\sum_{i\in A-\{t\}}(\;\boldsymbol{\sigma}(i)-\boldsymbol{\sigma}(i+1)\;)\; , & \text{if $t\in A$}\;
\end{cases}
\\=
\begin{cases}
\boldsymbol{x}({}_{-C}\mathrm{T}^{(+)})+\boldsymbol{\sigma}(1)+\boldsymbol{\sigma}(t)+\sum_{i\in C-(A\dot\cup\{t\})}(\;\boldsymbol{\sigma}(i)-\boldsymbol{\sigma}(i+1)\;)\; , & \text{if $t\not\in A$}\; ,\\
\boldsymbol{x}({}_{-C}\mathrm{T}^{(+)})+\sum_{i\in C-A}(\;\boldsymbol{\sigma}(i)-\boldsymbol{\sigma}(i+1)\;)\; , & \text{if $t\in A$}\; .
\end{cases}
\end{multline*}
\end{proof}

\section{Statistics on partitions of the negative parts of vertices of the hypercube graph and on the decompositions of vertices}

Recall that for any {\em odd\/} integer $\ell\in E_t$, we have
\begin{equation*}
|\{T\in\{1,-1\}^t\colon |\boldsymbol{Q}(T,\boldsymbol{R})|=\ell\}|=2\tbinom{t}{\ell}\; ,
\end{equation*}
see~\cite[Th.~13.6]{AM-PROM-I}.

Note also that $|\boldsymbol{Q}(\mathrm{T}^{(+)},\boldsymbol{R})|=|\boldsymbol{Q}(\mathrm{T}^{(-)},\boldsymbol{R})|=1$.

Now let $j\in [t-1]$. If
\begin{equation*}
j<\tfrac{\ell-1}{2}\ \ \ \text{or}\ \ \ j>t-\tfrac{\ell-1}{2}\; ,
\end{equation*}
then
\begin{equation*}
|\{T\in\{1,-1\}^t\colon |T^-|=j,\ |\boldsymbol{Q}(T,\boldsymbol{R})|=\ell\}|=0\; .
\end{equation*}
If
\begin{equation*}
\tfrac{\ell-1}{2}\leq j\leq t-\tfrac{\ell-1}{2}\; ,
\end{equation*}
then
\begin{gather*}
|\{T\in\{1,-1\}^t\colon |T^-|=j,\ |\boldsymbol{Q}(T,\boldsymbol{R})|=\ell\}|\\=
|\{T\in\{1,-1\}^t\colon |T^-|=t-j,\ |\boldsymbol{Q}(T,\boldsymbol{R})|=\ell\}|\\=
\binom{j-1}{\frac{\ell-1}{2}}\binom{t-j}{\frac{\ell-1}{2}}+
\binom{t-j-1}{\frac{\ell-1}{2}}\binom{j}{\frac{\ell-1}{2}}\; ,
\end{gather*}
see~\cite[Th.~2.7]{M-SC-II}.

Given {\em odd\/} integers $\ell',\ell'',\ell\in E_t$ and certain positive integers $j'$ and $j''$, in this section we obtain statistics related to  the following family of {\em ordered pairs\/} $(A,B)$ of disjoint {\em unordered subsets\/} $A$ and $B$ of the ground set $E_t$:
\begin{multline}
\label{eq:1}
\bigl\{(A,B)\in\mathbf{2}^{[t]}\times\mathbf{2}^{[t]}
\colon\ \ |A\cap B|=0,\;\ \ 0<|A|=:j'<t,\;\ \  0<|B|=:j''<t,\\  j'+j''<t,
\\ |\boldsymbol{Q}({}_{-A}\mathrm{T}^{(+)},\boldsymbol{R})|=\ell',\;\ \
|\boldsymbol{Q}({}_{-B}\mathrm{T}^{(+)},\boldsymbol{R})|=\ell'',\;\ \
 |\boldsymbol{Q}({}_{-(A\dot\cup B)}\mathrm{T}^{(+)},\boldsymbol{R})|=\ell\bigr\}\; .
\end{multline}

\begin{theorem}
\label{th:3}
\begin{itemize}
\item[\rm(i)]
In the family~{\rm(\ref{eq:1})} there are
\begin{multline}
\label{eq:7}
\binom{t-(j'+j'')-1}{\frac{\ell-1}{2}}\binom{j'-1}{\frac{\ell'-3}{2}}\binom{j''-1}{\frac{\ell''-3}{2}}\\
\times\begin{cases}
\left(\substack{\frac{\ell-1}{2}\\ \frac{\ell+\ell'-\ell''-1}{4}}\right)\!
\left(\substack{\frac{\ell+\ell'+\ell''-7}{4}\\ \frac{\ell-3}{2}}\right)\; , & \text{if $\frac{\ell+\ell'+\ell''-1}{2}$
odd\; ,}\\
\quad\\
\frac{\ell+\ell'-\ell''+1}{2}
\left(\substack{\frac{\ell-1}{2}\\ \frac{\ell+\ell'-\ell''+1}{4}}\right)\!
\left(\substack{\frac{\ell+\ell'+\ell''-9}{4}\\ \frac{\ell-3}{2}}\right)\; , & \text{if $\frac{\ell+\ell'+\ell''-1}{2}$
even}
\end{cases}
\end{multline}
pairs~$(A,B)$ of sets $A$ and $B$ such that
\begin{equation}
\label{eq:4}
|\{1,t\}\cap A|=|\{1,t\}\cap B|=0\; .
\end{equation}

\item[\rm(ii)]
In the family~{\rm(\ref{eq:1})} there are
\begin{multline}
\label{eq:8}
\binom{t-(j'+j'')-1}{\frac{\ell-3}{2}}\binom{j'-1}{\frac{\ell'-1}{2}}\binom{j''-1}{\frac{\ell''-3}{2}}\\
\times
\begin{cases}
\left(\substack{
\frac{\ell'-1}{2}
\\
\frac{\ell+\ell'-\ell''-1}{4}
}\right)\!
\left(\substack{
\frac{\ell+\ell'+\ell''-7}{4}
\\
\frac{\ell'-3}{2}
}\right)\; , & \text{if $\frac{\ell+\ell'+\ell''-1}{2}$
odd\; ,}\\
\quad\\
\frac{\ell+\ell'-\ell''+1}{2}
\left(\substack{
\frac{\ell'-1}{2}
\\
\frac{\ell+\ell'-\ell''+1}{4}
}\right)\!
\left(\substack{
\frac{\ell+\ell'+\ell''-9}{4}
\\
\frac{\ell'-3}{2}
}\right)\; , & \text{if $\frac{\ell+\ell'+\ell''-1}{2}$
even}
\end{cases}
\end{multline}
pairs~$(A,B)$ of sets $A$ and $B$ such that
\begin{equation}
\label{eq:6}
\{1,t\}\cap A=\{1,t\}\ \ \ \text{and}\ \ \ |\{1,t\}\cap B|=0\; .
\end{equation}

\item[\rm(iii)]
In the family~{\rm(\ref{eq:1})} there are
\begin{multline}
\label{eq:9}
\binom{t-(j'+j'')-1}{\frac{\ell-1}{2}}
\binom{j'-1}{\frac{\ell'-3}{2}}
\binom{j''-1}{\frac{\ell''-1}{2}}\\ \times
\begin{cases}
\frac{\ell+\ell''-\ell'+3}{2}
\left(\substack{
\frac{\ell-1}{2}
\\
\frac{\ell+\ell''-\ell'+1}{4}
}\right)\!
\left(\substack{
\frac{\ell+\ell'+\ell''-5}{4}
\\
\frac{\ell-1}{2}
}\right)\; ,\ \ \ \text{if $\frac{\ell+\ell'+\ell''+1}{2}$
odd\; ,}\\
\quad\\
\left(\substack{
\frac{\ell-1}{2}
\\
\frac{\ell+\ell''-\ell'-1}{4}
}\right)\!
\left(\substack{
\frac{\ell+\ell'+\ell''-3}{4}
\\
\frac{\ell-1}{2}
}\right)
+\frac{\ell+\ell''-\ell'+3}{2}
\left(\substack{
\frac{\ell-1}{2}
\\
\frac{\ell+\ell''-\ell'+3}{4}
}\right)\!
\left(\substack{
\frac{\ell+\ell'+\ell''-7}{4}
\\
\frac{\ell-1}{2}
}\right)\; ,\\
\hspace{6.4cm}\text{if $\frac{\ell+\ell'+\ell''+1}{2}$
even}
\end{cases}
\end{multline}
pairs~$(A,B)$ of sets~$A$ and~$B$ such that
\begin{equation}
\label{eq:5}
|\{1,t\}\cap A|=0\ \ \ \text{and} \ \ \ \{1,t\}\cap B=\{t\}\; .
\end{equation}

\item[\rm(iv)]
In the family~{\rm(\ref{eq:1})} there are
\begin{multline}
\label{eq:10}
\binom{t-(j'+j'')-1}{\frac{\ell-1}{2}}
\binom{j'-1}{\frac{\ell'-1}{2}}
\binom{j''-1}{\frac{\ell''-3}{2}}\\ \times
\begin{cases}
\frac{\ell+\ell'-\ell''+3}{2}
\left(\substack{
\frac{\ell'-1}{2}
\\
\frac{\ell+\ell'-\ell''+1}{4}
}\right)\!
\left(\substack{
\frac{\ell+\ell'+\ell''-5}{4}
\\
\frac{\ell'-1}{2}
}\right)\; ,\ \ \ \ \ \ \ \ \text{if $\frac{\ell+\ell'+\ell''+1}{2}$
odd\; ,}\\
\quad\\
\left(\substack{
\frac{\ell'-1}{2}
\\
\frac{\ell+\ell'-\ell''-1}{4}
}\right)\!
\left(\substack{
\frac{\ell+\ell'+\ell''-3}{4}
\\
\frac{\ell'-1}{2}
}\right)
+\frac{\ell+\ell'-\ell''+3}{2}
\left(\substack{
\frac{\ell'-1}{2}
\\
\frac{\ell+\ell'-\ell''+3}{4}
}\right)\!
\left(\substack{
\frac{\ell+\ell'+\ell''-7}{4}
\\
\frac{\ell'-1}{2}
}\right)\; ,
\\ \hspace{7cm}\!\text{if $\frac{\ell+\ell'+\ell''+1}{2}$
even}
\end{cases}
\end{multline}
pairs~$(A,B)$ of sets~$A$ and~$B$ such that
\begin{equation}
\label{eq:2}
\{1,t\}\cap A=\{1\}\ \ \ \text{and}\ \ \ |\{1,t\}\cap B|=0\; .
\end{equation}

\item[\rm(v)]
In the family~{\rm(\ref{eq:1})} there are
\begin{multline}
\label{eq:11}
\binom{t-(j'+j'')-1}{\frac{\ell-3}{2}}
\binom{j'-1}{\frac{\ell'-1}{2}}
\binom{j''-1}{\frac{\ell''-1}{2}}\\ \times
\begin{cases}
\frac{\ell'+\ell''-\ell+3}{2}
\left(\substack{
\frac{\ell'-1}{2}
\\
\frac{\ell'+\ell''-\ell+1}{4}
}\right)\!
\left(\substack{
\frac{\ell+\ell'+\ell''-5}{4}
\\
\frac{\ell'-1}{2}
}\right)\; ,\ \ \ \  \text{if $\frac{\ell+\ell'+\ell''+1}{2}$
odd\; ,}\\
\quad\\
\left(\substack{
\frac{\ell'-1}{2}
\\
\frac{\ell'+\ell''-\ell-1}{4}
}\right)\!
\left(\substack{
\frac{\ell+\ell'+\ell''-3}{4}
\\
\frac{\ell'-1}{2}
}\right)
+\frac{\ell'+\ell''-\ell+3}{2}
\left(\substack{
\frac{\ell'-1}{2}
\\
\frac{\ell'+\ell''-\ell+3}{4}
}\right)\!
\left(\substack{
\frac{\ell+\ell'+\ell''-7}{4}
\\
\frac{\ell'-1}{2}
}\right)\; ,
\\ \hspace{6.4cm}\text{if $\frac{\ell+\ell'+\ell''+1}{2}$
even}
\end{cases}
\end{multline}
pairs~$(A,B)$ of sets~$A$ and~$B$ such that
\begin{equation}
\label{eq:3}
\{1,t\}\cap A=\{1\}\ \ \ \text{and}\ \ \ \{1,t\}\cap B=\{t\}\; .
\end{equation}

\item[\rm(vi)]
In the family~{\rm(\ref{eq:1})} there are
\begin{multline*}
\binom{t-(j'+j'')-1}{\frac{\ell-3}{2}}\binom{j'-1}{\frac{\ell'-3}{2}}\binom{j''-1}{\frac{\ell''-1}{2}}\\
\times
\begin{cases}
\left(\substack{
\frac{\ell''-1}{2}
\\
\frac{\ell+\ell''-\ell'-1}{4}
}\right)\!
\left(\substack{
\frac{\ell+\ell'+\ell''-7}{4}
\\
\frac{\ell''-3}{2}
}\right)\; , & \text{if $\frac{\ell+\ell'+\ell''-1}{2}$
odd\; ,}\\
\quad\\
\frac{\ell+\ell''-\ell'+1}{2}
\left(\substack{
\frac{\ell''-1}{2}
\\
\frac{\ell+\ell''-\ell'+1}{4}
}\right)\!
\left(\substack{
\frac{\ell+\ell'+\ell''-9}{4}
\\
\frac{\ell''-3}{2}
}\right)\; , & \text{if $\frac{\ell+\ell'+\ell''-1}{2}$
even}
\end{cases}
\end{multline*}
pairs~$(A,B)$ of sets $A$ and $B$ such that
\begin{equation*}
|\{1,t\}\cap A|=0 \ \ \ \text{and}\ \ \ \{1,t\}\cap B=\{1,t\}\; .
\end{equation*}

\item[\rm(vii)]
In the family~{\rm(\ref{eq:1})} there are
\begin{multline*}
\binom{t-(j'+j'')-1}{\frac{\ell-1}{2}}\binom{j'-1}{\frac{\ell'-1}{2}}\binom{j''-1}{\frac{\ell''-3}{2}}\\ \times
\begin{cases}
\frac{\ell+\ell'-\ell''+3}{2}
\left(\substack{
\frac{\ell-1}{2}
\\
\frac{\ell+\ell'-\ell''+1}{4}
}\right)\!
\left(\substack{
\frac{\ell+\ell'+\ell''-5}{4}
\\
\frac{\ell-1}{2}
}\right)\; ,\ \ \ \text{if $\frac{\ell+\ell'+\ell''+1}{2}$
odd\; ,}\\
\quad\\
\left(\substack{
\frac{\ell-1}{2}
\\
\frac{\ell+\ell'-\ell''-1}{4}
}\right)\!
\left(\substack{
\frac{\ell+\ell'+\ell''-3}{4}
\\
\frac{\ell-1}{2}
}\right)
+\frac{\ell+\ell'-\ell''+3}{2}
\left(\substack{
\frac{\ell-1}{2}
\\
\frac{\ell+\ell'-\ell''+3}{4}
}\right)\!
\left(\substack{
\frac{\ell+\ell'+\ell''-7}{4}
\\
\frac{\ell-1}{2}
}\right)\; ,\\
\hspace{6.4cm}\text{if $\frac{\ell+\ell'+\ell''+1}{2}$
even}
\end{cases}
\end{multline*}
pairs~$(A,B)$ of sets~$A$ and~$B$ such that
\begin{equation*}
\{1,t\}\cap A=\{t\}\ \ \ \text{and} \ \ \ |\{1,t\}\cap B|=0\; .
\end{equation*}

\item[\rm(viii)]
In the family~{\rm(\ref{eq:1})} there are
\begin{multline*}
\binom{t-(j'+j'')-1}{\frac{\ell-1}{2}}\binom{j'-1}{\frac{\ell'-3}{2}}\binom{j''-1}{\frac{\ell''-1}{2}}\\ \times
\begin{cases}
\frac{\ell+\ell''-\ell'+3}{2}
\left(\substack{
\frac{\ell''-1}{2}
\\
\frac{\ell+\ell''-\ell'+1}{4}
}\right)\!
\left(\substack{
\frac{\ell+\ell'+\ell''-5}{4}
\\
\frac{\ell''-1}{2}
}\right)\; ,\ \ \ \ \ \ \ \ \text{if $\frac{\ell+\ell'+\ell''+1}{2}$
odd\; ,}\\
\quad\\
\left(\substack{
\frac{\ell''-1}{2}
\\
\frac{\ell+\ell''-\ell'-1}{4}
}\right)\!
\left(\substack{
\frac{\ell+\ell'+\ell''-3}{4}
\\
\frac{\ell''-1}{2}
}\right)
+\frac{\ell+\ell''-\ell'+3}{2}
\left(\substack{
\frac{\ell''-1}{2}
\\
\frac{\ell+\ell''-\ell'+3}{4}
}\right)\!
\left(\substack{
\frac{\ell+\ell'+\ell''-7}{4}
\\
\frac{\ell''-1}{2}
}\right)\; ,
\\ \hspace{7cm}\!\text{if $\frac{\ell+\ell'+\ell''+1}{2}$
even}
\end{cases}
\end{multline*}
pairs~$(A,B)$ of sets~$A$ and~$B$ such that
\begin{equation*}
|\{1,t\}\cap A|=0\ \ \ \text{and}\ \ \  \{1,t\}\cap B=\{1\}\; .
\end{equation*}

\item[\rm(ix)]
In the family~{\rm(\ref{eq:1})} there are
\begin{multline*}
\binom{t-(j'+j'')-1}{\frac{\ell-3}{2}}\binom{j'-1}{\frac{\ell'-1}{2}}\binom{j''-1}{\frac{\ell''-1}{2}}\\ \times
\begin{cases}
\frac{\ell'+\ell''-\ell+3}{2}
\left(\substack{
\frac{\ell''-1}{2}
\\
\frac{\ell'+\ell''-\ell+1}{4}
}\right)\!
\left(\substack{
\frac{\ell+\ell'+\ell''-5}{4}
\\
\frac{\ell''-1}{2}
}\right)\; ,\ \ \ \  \text{if $\frac{\ell+\ell'+\ell''+1}{2}$
odd\; ,}\\
\quad\\
\left(\substack{
\frac{\ell''-1}{2}
\\
\frac{\ell'+\ell''-\ell-1}{4}
}\right)\!
\left(\substack{
\frac{\ell+\ell'+\ell''-3}{4}
\\
\frac{\ell''-1}{2}
}\right)
+\frac{\ell'+\ell''-\ell+3}{2}
\left(\substack{
\frac{\ell''-1}{2}
\\
\frac{\ell'+\ell''-\ell+3}{4}
}\right)\!
\left(\substack{
\frac{\ell+\ell'+\ell''-7}{4}
\\
\frac{\ell''-1}{2}
}\right)\; ,
\\ \hspace{6.4cm}\text{if $\frac{\ell+\ell'+\ell''+1}{2}$
even}
\end{cases}
\end{multline*}
pairs~$(A,B)$ of sets~$A$ and~$B$ such that
\begin{equation*}
\{1,t\}\cap A=\{t\}\ \ \ \text{and}\ \ \ \{1,t\}\cap B=\{1\}\; .
\end{equation*}
\end{itemize}
\end{theorem}

Before proceeding to the proof of the theorem, recall that the {\em Smirnov words\/} (see Appendix on page~\pageref{appendixpage}) are defined to be the words, any two consecutive letters of which are distinct.

Let $(\theta,\alpha,\beta)$ be a three-letter alphabet.

Given two letters $\mathfrak{s}'\in(\theta,\alpha,\beta)$ and $\mathfrak{s}''\in(\theta,\alpha,\beta)$, we denote by
\begin{equation*}
\mathfrak{T}(\mathfrak{s}', \mathfrak{s}'';n(\theta),n(\alpha),n(\beta))
\end{equation*}
the number of ternary Smirnov words (that start with $\mathfrak{s}'$ and end with $\mathfrak{s}''$)
with exactly $n(\theta)$ letters $\theta$, with $n(\alpha)$ letters $\alpha$, and with~$n(\beta)$ letters $\beta$.

We denote by $\mathtt{c}(m;n)$, where $\mathtt{c}(m;n)=\tbinom{n-1}{m-1}$, the number of {\em compositions\/} of a positive integer~$n$ with $m$ positive parts.

We will regard sets $A$ and $B$ composing pairs of the family~(\ref{eq:1})
as disjoint unions
\begin{align*}
A&=[i'_1,k'_1]\;\dot\cup\;[i'_2,k'_2]\;\dot\cup\;\cdots\;\dot\cup\;[i'_{\varrho(A)},k'_{\varrho(A)}]
\intertext{and}
B&=[i''_1,k''_1]\;\dot\cup\;[i''_2,k''_2]\;\dot\cup\;\cdots\;\dot\cup\;[i''_{\varrho(B)},k''_{\varrho(B)}]
\end{align*}
of intervals $[i,k]:=\{i,i+1,\ldots,k\}$ such that
\begin{equation*}
k'_1+2\leq i'_2,\ \ k'_2+2\leq i'_3,\ \ \ldots,\ \
k'_{\varrho(A)-1}+2\leq i'_{\varrho(A)}
\end{equation*}
and
\begin{equation*}
k''_1+2\leq i''_2,\ \ k''_2+2\leq i''_3,\ \ \ldots,\ \
k''_{\varrho(B)-1}+2\leq i''_{\varrho(B)}\; .
\end{equation*}

\begin{proof}
\noindent{\rm(i)} Let us count the number of pairs $(A,B)$ in the family~(\ref{eq:1}) such that
\begin{equation*}
|\{1,t\}\cap A|=|\{1,t\}\cap B|=0\; .
\end{equation*}
For any such a pair, by~\cite[Lem.~2.6(iii)]{M-SC-II}, we know that
\begin{equation*}
\varrho(A)=\tfrac{\ell'-1}{2}\ \ \ \text{and}\ \ \ \varrho(B)=\tfrac{\ell''-1}{2}\; ,
\end{equation*}
and the set $E_t-(A\dot\cup B)$ is a disjoint union of $\tfrac{\ell+1}{2}$ intervals.

For each pair~$(A,B)$, pick an arbitrary {\em system\/} (arranged in {\em ascending order}) of {\em distinct representatives\/} $(e_1<e_2<\cdots<e_{(\ell+\ell'+\ell''-1)/2})$ of the intervals composing the sets $A$, $B$ and $E_t-(A\dot\cup B)$.
By making the substitutions
\begin{equation*}
e_i\mapsto\begin{cases}
\theta\; , & \text{if $e_i\in E_t-(A\dot\cup B)$\; ,}\\
\alpha\; , & \text{if $e_i\in A$\; ,}\\
\beta\; , & \text{if $e_i\in B$\; ,}
\end{cases}\ \ \ \ \ 1\leq i\leq (\ell+\ell'+\ell''-1)/2\; ,
\end{equation*}
for all of the pairs, we get
\begin{equation*}
\mathfrak{T}\left(\theta,\theta; \tfrac{\ell+1}{2}, \tfrac{\ell'-1}{2}, \tfrac{\ell''-1}{2}\right)
\end{equation*}
different {\em ternary Smirnov words}, of length $(\ell+\ell'+\ell''-1)/2$, that start with~$\theta$,
end with
$\theta$, and contain exactly $\tfrac{\ell+1}{2}$ letters~$\theta$, $\tfrac{\ell'-1}{2}$ letters~$\alpha$, and~$\tfrac{\ell''-1}{2}$ letters~$\beta$.

By Remark~\ref{th:2}(i), we see that
\begin{multline*}
\mathfrak{T}\left(\theta,\theta; \tfrac{\ell+1}{2}, \tfrac{\ell'-1}{2}, \tfrac{\ell''-1}{2}\right)\\=
\begin{cases}
\left(\substack{\frac{\ell+1}{2}-1\\ \frac{\frac{\ell+1}{2}+\frac{\ell'-1}{2}-\frac{\ell''-1}{2}-1}{2}}\right)\!
\left(\substack{\frac{\frac{\ell+1}{2}+\frac{\ell'-1}{2}+\frac{\ell''-1}{2}-3}{2}\\ \frac{\ell+1}{2}-2}\right)\; ,\ \ \ \  \text{if $\frac{\ell+\ell'+\ell''-1}{2}$
odd\; ,}\\
\quad\\
(\frac{\ell+1}{2}+\frac{\ell'-1}{2}-\frac{\ell''-1}{2})\cdot
\left(\substack{\frac{\ell+1}{2}-1\\ \frac{\frac{\ell+1}{2}+\frac{\ell'-1}{2}-\frac{\ell''-1}{2}}{2}}\right)\!
\left(\substack{\frac{\frac{\ell+1}{2}+\frac{\ell'-1}{2}+\frac{\ell''-1}{2}}{2}-2\\ \frac{\ell+1}{2}-2}\right)\; ,\\  \hspace{7.3cm}\text{if $\frac{\ell+\ell'+\ell''-1}{2}$
even\; ,}
\end{cases}
\end{multline*}
that is,
\begin{multline*}
\mathfrak{T}\left(\theta,\theta; \tfrac{\ell+1}{2}, \tfrac{\ell'-1}{2}, \tfrac{\ell''-1}{2}\right)\\=
\begin{cases}
\left(\substack{\frac{\ell-1}{2}\\ \frac{\ell+\ell'-\ell''-1}{4}}\right)\!
\left(\substack{\frac{\ell+\ell'+\ell''-7}{4}\\ \frac{\ell-3}{2}}\right)\; , & \text{if $\frac{\ell+\ell'+\ell''-1}{2}$
odd\; ,}\\
\quad\\
\frac{\ell+\ell'-\ell''+1}{2}
\left(\substack{\frac{\ell-1}{2}\\ \frac{\ell+\ell'-\ell''+1}{4}}\right)\!
\left(\substack{\frac{\ell+\ell'+\ell''-9}{4}\\ \frac{\ell-3}{2}}\right)\; , & \text{if $\frac{\ell+\ell'+\ell''-1}{2}$
even\; .}
\end{cases}
\end{multline*}
Since there are
\begin{equation*}
\mathfrak{T}\left(\theta,\theta;\tfrac{\ell+1}{2},\tfrac{\ell'-1}{2}, \tfrac{\ell''-1}{2}\right)
\cdot\mathtt{c}\bigl(\tfrac{\ell+1}{2};t-(j'+j'')\bigr)\cdot
\mathtt{c}(\tfrac{\ell'-1}{2};j')
\cdot\mathtt{c}(\tfrac{\ell''-1}{2};j'')
\end{equation*}
pairs~$(A,B)$ of sets $A$ and $B$ with the properties given in~(\ref{eq:4}), we see that the number of these pairs
in the family~(\ref{eq:1}) can be calculated by means of~(\ref{eq:7}).

\noindent{\rm(ii)}
Let us count the number of pairs $(A,B)$ in the family~(\ref{eq:1}) such that
\begin{equation*}
\{1,t\}\cap A=\{1,t\}\ \ \ \text{and}\ \ \ |\{1,t\}\cap B|=0\; .
\end{equation*}
For any such a pair, by~\cite[Lem.~2.6(ii), (iii)]{M-SC-II}, we know that
\begin{equation*}
\varrho(A)=\tfrac{\ell'+1}{2}\ \ \ \text{and}\ \ \ \varrho(B)=\tfrac{\ell''-1}{2}\; ,
\end{equation*}
and the set $E_t-(A\dot\cup B)$ is a disjoint union of $\tfrac{\ell-1}{2}$ intervals. We denote by
\begin{equation*}
\mathfrak{T}\left(\alpha,\alpha;\tfrac{\ell-1}{2}, \tfrac{\ell'+1}{2}, \tfrac{\ell''-1}{2}\right)
\end{equation*}
the number of ternary Smirnov words, of length $(\ell+\ell'+\ell''-1)/2$, that start with
$\alpha$,
end with
$\alpha$, and contain $\tfrac{\ell-1}{2}$ letters~$\theta$, $\tfrac{\ell'+1}{2}$ letters~$\alpha$, and $\tfrac{\ell''-1}{2}$ letters~$\beta$; in the family~(\ref{eq:1}) there are
\begin{equation*}
\mathfrak{T}\left(\alpha,\alpha;\tfrac{\ell-1}{2}, \tfrac{\ell'+1}{2}, \tfrac{\ell''-1}{2}\right)\cdot
\mathtt{c}\bigl(\tfrac{\ell-1}{2};t-(j'+j'')\bigr)\cdot
\mathtt{c}(\tfrac{\ell'+1}{2};j')
\cdot\mathtt{c}(\tfrac{\ell''-1}{2};j'')
\end{equation*}
pairs of sets $A$ and $B$ with the properties given in~(\ref{eq:6}).

By analogy with expression~(\ref{eq:15}), the number $\mathfrak{T}(\alpha,\alpha;n(\theta),n(\alpha),n(\beta))$ of ternary Smirnov words that start with $\alpha$, end with $\alpha$, and contain $n(\theta)$ letters $\theta$, $n(\alpha)$ letters $\alpha$, and $n(\beta)$ letters $\beta$, is
\begin{multline*}
\mathfrak{T}(\alpha,\alpha;n(\theta),n(\alpha),n(\beta))\\=
\begin{cases}
\left(\substack{
n(\alpha)-1
\\
\frac{n(\alpha)+n(\theta)-n(\beta)-1}{2}
}\right)\!
\left(\substack{
\frac{n(\alpha)+n(\theta)+n(\beta)-3}{2}
\\
n(\alpha)-2
}\right)
\; ,\ \ \ \  \text{if $n(\alpha)+n(\theta)+n(\beta)$
odd\; ,}\\
\quad\\
(n(\alpha)+n(\theta)-n(\beta))\cdot
\left(\substack{
n(\alpha)-1
\\
\frac{n(\alpha)+n(\theta)-n(\beta)}{2}
}\right)\!
\left(\substack{
\frac{n(\alpha)+n(\theta)+n(\beta)}{2}-2
\\
n(\alpha)-2
}\right)\; ,\\  \hspace{6.5cm}\text{if $n(\alpha)+n(\theta)+n(\beta)$
even\; .}
\end{cases}
\end{multline*}
As a consequence,
\begin{multline*}
\mathfrak{T}\left(\alpha,\alpha; \tfrac{\ell-1}{2}, \tfrac{\ell'+1}{2}, \tfrac{\ell''-1}{2}\right)\\=
\begin{cases}
\left(\substack{
\frac{\ell'+1}{2}-1
\\
\frac{\frac{\ell'+1}{2}+\frac{\ell-1}{2}-\frac{\ell''-1}{2}-1}{2}
}\right)\!
\left(\substack{
\frac{\frac{\ell'+1}{2}+\frac{\ell-1}{2}+\frac{\ell''-1}{2}-3}{2}
\\
\frac{\ell'+1}{2}-2
}\right)\; ,\ \ \ \  \text{if $\frac{\ell'+1}{2}+\frac{\ell-1}{2}+\frac{\ell''-1}{2}$
odd\; ,}\\
\quad\\
(\frac{\ell'+1}{2}+\frac{\ell-1}{2}-\frac{\ell''-1}{2})\cdot
\left(\substack{
\frac{\ell'+1}{2}-1
\\
\frac{\frac{\ell'+1}{2}+\frac{\ell-1}{2}-\frac{\ell''-1}{2}}{2}
}\right)\!
\left(\substack{
\frac{\frac{\ell'+1}{2}+\frac{\ell-1}{2}+\frac{\ell''-1}{2}}{2}-2
\\
\frac{\ell'+1}{2}-2
}\right)\; ,\\ \hspace{7.3cm}\text{if $\frac{\ell'+1}{2}+\frac{\ell-1}{2}+\frac{\ell''-1}{2}$
even\; ,}
\end{cases}
\end{multline*}
that is,
\begin{multline*}
\mathfrak{T}\left(\alpha,\alpha; \tfrac{\ell-1}{2}, \tfrac{\ell'+1}{2}, \tfrac{\ell''-1}{2}\right)\\=
\begin{cases}
\left(\substack{
\frac{\ell'-1}{2}
\\
\frac{\ell+\ell'-\ell''-1}{4}
}\right)\!
\left(\substack{
\frac{\ell+\ell'+\ell''-7}{4}
\\
\frac{\ell'-3}{2}
}\right)\; , & \text{if $\frac{\ell+\ell'+\ell''-1}{2}$
odd\; ,}\\
\quad\\
\frac{\ell+\ell'-\ell''+1}{2}
\left(\substack{
\frac{\ell'-1}{2}
\\
\frac{\ell+\ell'-\ell''+1}{4}
}\right)\!
\left(\substack{
\frac{\ell+\ell'+\ell''-9}{4}
\\
\frac{\ell'-3}{2}
}\right)\; , & \text{if $\frac{\ell+\ell'+\ell''-1}{2}$
even\; .}
\end{cases}
\end{multline*}
Thus, the number of pairs~$(A,B)$ of sets $A$ and $B$ in the family~(\ref{eq:1}), with the properties given in~(\ref{eq:6}), can be calculated by means of~(\ref{eq:8}).

\noindent{\rm(iii)}
Let us consider the pairs $(A,B)$ in the family~(\ref{eq:1}) such that
\begin{equation*}
|\{1,t\}\cap A|=0\ \ \ \text{and} \ \ \ \{1,t\}\cap B=\{t\}\; .
\end{equation*}
For any such a pair, by~\cite[Lem.~2.6(iii), (i)(b)]{M-SC-II}, we know that
\begin{equation*}
\varrho(A)=\tfrac{\ell'-1}{2}\ \ \ \text{and}\ \ \ \varrho(B)=\tfrac{\ell''+1}{2}\; ,
\end{equation*}
and the set $E_t-(A\dot\cup B)$ is a disjoint union of $\tfrac{\ell+1}{2}$ intervals. We denote by
\begin{equation*}
\mathfrak{T}\left(\theta,\beta;\tfrac{\ell+1}{2},\tfrac{\ell'-1}{2},\tfrac{\ell''+1}{2}\right)
\end{equation*}
the number of ternary Smirnov words, of length $(\ell+\ell'+\ell''+1)/2$, that start with
$\theta$,
end with
$\beta$, and contain $\tfrac{\ell+1}{2}$ letters~$\theta$, $\tfrac{\ell'-1}{2}$ letters~$\alpha$, and $\tfrac{\ell''+1}{2}$ letters~$\beta$; in the family~(\ref{eq:1}) there are
\begin{equation*}
\mathfrak{T}\left(\theta,\beta;\tfrac{\ell+1}{2},\tfrac{\ell'-1}{2},\tfrac{\ell''+1}{2}\right)\cdot
\mathtt{c}\bigl(\tfrac{\ell+1}{2};t-(j'+j'')\bigr)\cdot
\mathtt{c}(\tfrac{\ell'-1}{2};j')
\cdot\mathtt{c}(\tfrac{\ell''+1}{2};j'')
\end{equation*}
pairs of sets $A$ and $B$ with the properties given in~(\ref{eq:5}).

By Remark~\ref{th:2}(ii), we have
\begin{multline*}
\mathfrak{T}\left(\theta,\beta;\tfrac{\ell+1}{2},\tfrac{\ell'-1}{2},\tfrac{\ell''+1}{2}\right)\\=
\begin{cases}
(\frac{\ell+1}{2}+\frac{\ell''+1}{2}-\frac{\ell'-1}{2})\cdot
\left(\substack{
\frac{\ell+1}{2}-1
\\
\frac{\frac{\ell+1}{2}+\frac{\ell''+1}{2}-\frac{\ell'-1}{2}-1}{2}
}\right)\!
\left(\substack{
\frac{\frac{\ell+1}{2}+\frac{\ell'-1}{2}+\frac{\ell''+1}{2}-3}{2}
\\
\frac{\ell+1}{2}-1
}\right)\; ,\\  \hspace{6cm}\text{if $\frac{\ell+1}{2}+\frac{\ell'-1}{2}+\frac{\ell''+1}{2}$
odd\; ,}\\
\quad\\
\left(\substack{
\frac{\ell+1}{2}-1
\\
\frac{\frac{\ell+1}{2}+\frac{\ell''+1}{2}-\frac{\ell'-1}{2}}{2}-1
}\right)\!
\left(\substack{
\frac{\frac{\ell+1}{2}+\frac{\ell'-1}{2}+\frac{\ell''+1}{2}}{2}-1
\\
\frac{\ell+1}{2}-1
}\right)\\
+\,(\frac{\ell+1}{2}+\frac{\ell''+1}{2}-\frac{\ell'-1}{2})\cdot
\left(\substack{
\frac{\ell+1}{2}-1
\\
\frac{\frac{\ell+1}{2}+\frac{\ell''+1}{2}-\frac{\ell'-1}{2}}{2}
}\right)\!
\left(\substack{
\frac{\frac{\ell+1}{2}+\frac{\ell'-1}{2}+\frac{\ell''+1}{2}}{2}-2
\\
\frac{\ell+1}{2}-1
}\right)\; ,\\
\hspace{6cm}\text{if $\frac{\ell+1}{2}+\frac{\ell'-1}{2}+\frac{\ell''+1}{2}$
even\; ,}
\end{cases}
\end{multline*}
that is,
\begin{multline*}
\mathfrak{T}\left(\theta,\beta;\tfrac{\ell+1}{2},\tfrac{\ell'-1}{2},\tfrac{\ell''+1}{2}\right)\\=
\begin{cases}
\frac{\ell+\ell''-\ell'+3}{2}
\left(\substack{
\frac{\ell-1}{2}
\\
\frac{\ell+\ell''-\ell'+1}{4}
}\right)\!
\left(\substack{
\frac{\ell+\ell'+\ell''-5}{4}
\\
\frac{\ell-1}{2}
}\right)\; ,\ \ \ \text{if $\frac{\ell+\ell'+\ell''+1}{2}$
odd\; ,}\\
\quad\\
\left(\substack{
\frac{\ell-1}{2}
\\
\frac{\ell+\ell''-\ell'-1}{4}
}\right)\!
\left(\substack{
\frac{\ell+\ell'+\ell''-3}{4}
\\
\frac{\ell-1}{2}
}\right)
+\frac{\ell+\ell''-\ell'+3}{2}
\left(\substack{
\frac{\ell-1}{2}
\\
\frac{\ell+\ell''-\ell'+3}{4}
}\right)\!
\left(\substack{
\frac{\ell+\ell'+\ell''-7}{4}
\\
\frac{\ell-1}{2}
}\right)\; ,\\
\hspace{6.4cm}\text{if $\frac{\ell+\ell'+\ell''+1}{2}$
even\; .}
\end{cases}
\end{multline*}
We see that the number of pairs~$(A,B)$ of sets $A$ and $B$ in the family~(\ref{eq:1}), with the properties given in~(\ref{eq:5}),
can be calculated by means of~(\ref{eq:9}).

\noindent{\rm(iv)}
Let us consider the pairs $(A,B)$ in the family~(\ref{eq:1}) such that
\begin{equation*}
\{1,t\}\cap A=\{1\}\ \ \ \text{and}\ \ \ |\{1,t\}\cap B|=0\; .
\end{equation*}
For any such a pair, by~\cite[Lem.~2.6(i)(a),(iii)]{M-SC-II}, we have
\begin{equation*}
\varrho(A)=\tfrac{\ell'+1}{2}\ \ \ \text{and}\ \ \ \varrho(B)=\tfrac{\ell''-1}{2}\; ;
\end{equation*}
note also that the set $E_t-(A\dot\cup B)$ is a disjoint union of $\tfrac{\ell+1}{2}$ intervals.

We denote by
\begin{equation*}
\mathfrak{T}\left(\alpha,\theta;\tfrac{\ell+1}{2},\tfrac{\ell'+1}{2}, \tfrac{\ell''-1}{2}\right)
\end{equation*}
the number of ternary Smirnov words, of length $(\ell+\ell'+\ell''+1)/2$, that start with
$\alpha$,
end with
$\theta$, and contain exactly $\tfrac{\ell+1}{2}$ letters~$\theta$, $\tfrac{\ell'+1}{2}$ letters~$\alpha$, and $\tfrac{\ell''-1}{2}$ letters~$\beta$; in the family~(\ref{eq:1}) there are
\begin{equation*}
\mathfrak{T}\left(\alpha,\theta;\tfrac{\ell+1}{2},\tfrac{\ell'+1}{2}, \tfrac{\ell''-1}{2}\right)
\cdot\mathtt{c}\bigl(\tfrac{\ell+1}{2};t-(j'+j'')\bigr)\cdot
\mathtt{c}(\tfrac{\ell'+1}{2};j')
\cdot\mathtt{c}(\tfrac{\ell''-1}{2};j'')
\end{equation*}
pairs of sets $A$ and $B$ with the properties given in~(\ref{eq:2}).

By analogy with expression~(\ref{eq:16}), the number $\mathfrak{T}(\alpha,\theta;n(\theta),n(\alpha),n(\beta))$ of ternary Smirnov words that start with~$\alpha$, end with $\theta$, and contain $n(\theta)$ letters~$\theta$, $n(\alpha)$ letters $\alpha$, and $n(\beta)$ letters~$\beta$, is
\begin{multline*}
\mathfrak{T}(\alpha,\theta;n(\theta),n(\alpha),n(\beta))\\=
\begin{cases}
(n(\alpha)+n(\theta)-n(\beta))\cdot
\left(\substack{
n(\alpha)-1
\\
\frac{n(\alpha)+n(\theta)-n(\beta)-1}{2}
}\right)\!
\left(\substack{
\frac{n(\alpha)+n(\beta)+n(\theta)-3}{2}
\\
n(\alpha)-1
}\right)\; ,\\  \hspace{6cm}\text{if $n(\alpha)+n(\beta)+n(\theta)$
odd\; ,}\\
\quad\\
\left(\substack{
n(\alpha)-1
\\
\frac{n(\alpha)+n(\theta)-n(\beta)}{2}-1
}\right)\!
\left(\substack{
\frac{n(\alpha)+n(\beta)+n(\theta)}{2}-1
\\
n(\alpha)-1
}\right)\\
+\,(n(\alpha)+n(\theta)-n(\beta))\cdot
\left(\substack{
n(\alpha)-1
\\
\frac{n(\alpha)+n(\theta)-n(\beta)}{2}
}\right)\!
\left(\substack{
\frac{n(\alpha)+n(\beta)+n(\theta)}{2}-2
\\
n(\alpha)-1
}\right)\; ,\\
\hspace{6cm}\text{if $n(\alpha)+n(\beta)+n(\theta)$
even\; .}
\end{cases}
\end{multline*}
As a consequence,
\begin{multline*}
\mathfrak{T}\left(\alpha,\theta;\tfrac{\ell+1}{2},\tfrac{\ell'+1}{2}, \tfrac{\ell''-1}{2}\right)\\=
\begin{cases}
(\frac{\ell'+1}{2}+\frac{\ell+1}{2}-\frac{\ell''-1}{2})\cdot
\left(\substack{
\frac{\ell'+1}{2}-1
\\
\frac{\frac{\ell'+1}{2}+\frac{\ell+1}{2}-\frac{\ell''-1}{2}-1}{2}
}\right)\!
\left(\substack{
\frac{\frac{\ell'+1}{2}+\frac{\ell''-1}{2}+\frac{\ell+1}{2}-3}{2}
\\
\frac{\ell'+1}{2}-1
}\right)\; ,\\  \hspace{6cm}\text{if $\frac{\ell'+1}{2}+\frac{\ell''-1}{2}+\frac{\ell+1}{2}$
odd\; ,}\\
\quad\\
\left(\substack{
\frac{\ell'+1}{2}-1
\\
\frac{\frac{\ell'+1}{2}+\frac{\ell+1}{2}-\frac{\ell''-1}{2}}{2}-1
}\right)\!
\left(\substack{
\frac{\frac{\ell'+1}{2}+\frac{\ell''-1}{2}+\frac{\ell+1}{2}}{2}-1
\\
\frac{\ell'+1}{2}-1
}\right)\\
+\,(\frac{\ell'+1}{2}+\frac{\ell+1}{2}-\frac{\ell''-1}{2})\cdot
\left(\substack{
\frac{\ell'+1}{2}-1
\\
\frac{\frac{\ell'+1}{2}+\frac{\ell+1}{2}-\frac{\ell''-1}{2}}{2}
}\right)\!
\left(\substack{
\frac{\frac{\ell'+1}{2}+\frac{\ell''-1}{2}+\frac{\ell+1}{2}}{2}-2
\\
\frac{\ell'+1}{2}-1
}\right)\; ,\\
\hspace{6cm}\text{if $\frac{\ell'+1}{2}+\frac{\ell''-1}{2}+\frac{\ell+1}{2}$
even\; ,}
\end{cases}
\end{multline*}
that is,
\begin{multline*}
\mathfrak{T}\left(\alpha,\theta;\tfrac{\ell+1}{2},\tfrac{\ell'+1}{2}, \tfrac{\ell''-1}{2}\right)\\=
\begin{cases}
\frac{\ell+\ell'-\ell''+3}{2}
\left(\substack{
\frac{\ell'-1}{2}
\\
\frac{\ell+\ell'-\ell''+1}{4}
}\right)\!
\left(\substack{
\frac{\ell+\ell'+\ell''-5}{4}
\\
\frac{\ell'-1}{2}
}\right)\; ,\ \ \ \ \ \ \ \ \text{if $\frac{\ell+\ell'+\ell''+1}{2}$
odd\; ,}\\
\quad\\
\left(\substack{
\frac{\ell'-1}{2}
\\
\frac{\ell+\ell'-\ell''-1}{4}
}\right)\!
\left(\substack{
\frac{\ell+\ell'+\ell''-3}{4}
\\
\frac{\ell'-1}{2}
}\right)
+\frac{\ell+\ell'-\ell''+3}{2}
\left(\substack{
\frac{\ell'-1}{2}
\\
\frac{\ell+\ell'-\ell''+3}{4}
}\right)\!
\left(\substack{
\frac{\ell+\ell'+\ell''-7}{4}
\\
\frac{\ell'-1}{2}
}\right)\; ,
\\ \hspace{7cm}\!\text{if $\frac{\ell+\ell'+\ell''+1}{2}$
even\; .}
\end{cases}
\end{multline*}
We see that the number of pairs~$(A,B)$ of sets~$A$ and~$B$ in the family~(\ref{eq:1}), with the properties given in~(\ref{eq:2}),
can be calculated by means of~(\ref{eq:10}).

\noindent{\rm(v)}
Let us consider the pairs $(A,B)$ in the family~(\ref{eq:1}) such that
\begin{equation*}
\{1,t\}\cap A=\{1\}\ \ \ \text{and}\ \ \ \{1,t\}\cap B=\{t\}\; .
\end{equation*}
For any such a pair, by~\cite[Lem.~2.6(i)(a), (i)(b), (iii)]{M-SC-II}, we know that
\begin{equation*}
\varrho(A)=\tfrac{\ell'+1}{2}\ \ \ \text{and}\ \ \ \varrho(B)=\tfrac{\ell''+1}{2}\; ,
\end{equation*}
and the set $E_t-(A\dot\cup B)$ is a disjoint union of $\tfrac{\ell-1}{2}$ intervals.

We denote by
\begin{equation*}
\mathfrak{T}\left(\alpha,\beta;\tfrac{\ell-1}{2},\tfrac{\ell'+1}{2}, \tfrac{\ell''+1}{2}\right)
\end{equation*}
the number of ternary Smirnov words, of length $(\ell+\ell'+\ell''+1)/2$, that start with
$\alpha$, end with $\beta$, and contain $\tfrac{\ell-1}{2}$ letters~$\theta$, $\tfrac{\ell'+1}{2}$ letters~$\alpha$, and $\tfrac{\ell''+1}{2}$ letters~$\beta$; in the family~(\ref{eq:1}) there are
\begin{equation*}
\mathfrak{T}\left(\alpha,\beta; \tfrac{\ell-1}{2}, \tfrac{\ell'+1}{2}, \tfrac{\ell''+1}{2}\right)\cdot
\mathtt{c}\bigl(\tfrac{\ell-1}{2};t-(j'+j'')\bigr)\cdot
\mathtt{c}(\tfrac{\ell'+1}{2};j')
\cdot\mathtt{c}(\tfrac{\ell''+1}{2};j'')
\end{equation*}
pairs of sets $A$ and $B$ with the properties given in~(\ref{eq:3}).

By analogy with expression~(\ref{eq:16}), the number $\mathfrak{T}(\alpha,\beta;n(\theta),n(\alpha),n(\beta))$ of ternary Smirnov words that start with $\alpha$, end with $\beta$, and contain $n(\theta)$ letters $\theta$, $n(\alpha)$ letters $\alpha$ and $n(\beta)$ letters $\beta$, is
\begin{multline*}
\mathfrak{T}(\alpha,\beta;n(\theta),n(\alpha),n(\beta))\\=
\begin{cases}
(n(\alpha)+n(\beta)-n(\theta))\cdot
\left(\substack{
n(\alpha)-1
\\
\frac{n(\alpha)+n(\beta)-n(\theta)-1}{2}
}\right)\!
\left(\substack{
\frac{n(\alpha)+n(\theta)+n(\beta)-3}{2}
\\
n(\alpha)-1
}\right)\; ,\\  \hspace{6cm}\text{if $n(\alpha)+n(\theta)+n(\beta)$
odd\; ,}\\
\quad\\
\left(\substack{
n(\alpha)-1
\\
\frac{n(\alpha)+n(\beta)-n(\theta)}{2}-1
}\right)\!
\left(\substack{
\frac{n(\alpha)+n(\theta)+n(\beta)}{2}-1
\\
n(\alpha)-1
}\right)\\
+\,(n(\alpha)+n(\beta)-n(\theta))\cdot
\left(\substack{
n(\alpha)-1
\\
\frac{n(\alpha)+n(\beta)-n(\theta)}{2}
}\right)\!
\left(\substack{
\frac{n(\alpha)+n(\theta)+n(\beta)}{2}-2
\\
n(\alpha)-1
}\right)\; ,\\
\hspace{6cm}\text{if $n(\alpha)+n(\theta)+n(\beta)$
even\; .}
\end{cases}
\end{multline*}
As a consequence,
\begin{multline*}
\mathfrak{T}\left(\alpha,\beta; \tfrac{\ell-1}{2}, \tfrac{\ell'+1}{2}, \tfrac{\ell''+1}{2}\right)\\=
\begin{cases}
(\frac{\ell'+1}{2}+\frac{\ell''+1}{2}-\frac{\ell-1}{2})\cdot
\left(\substack{
\frac{\ell'+1}{2}-1
\\
\frac{\frac{\ell'+1}{2}+\frac{\ell''+1}{2}-\frac{\ell-1}{2}-1}{2}
}\right)\!
\left(\substack{
\frac{\frac{\ell'+1}{2}+\frac{\ell-1}{2}+\frac{\ell''+1}{2}-3}{2}
\\
\frac{\ell'+1}{2}-1
}\right)\; ,\\  \hspace{6cm}\text{if $\frac{\ell'+1}{2}+\frac{\ell-1}{2}+\frac{\ell''+1}{2}$
odd\; ,}\\
\quad\\
\left(\substack{
\frac{\ell'+1}{2}-1
\\
\frac{\frac{\ell'+1}{2}+\frac{\ell''+1}{2}-\frac{\ell-1}{2}}{2}-1
}\right)\!
\left(\substack{
\frac{\frac{\ell'+1}{2}+\frac{\ell-1}{2}+\frac{\ell''+1}{2}}{2}-1
\\
\frac{\ell'+1}{2}-1
}\right)\\
+\,(\frac{\ell'+1}{2}+\frac{\ell''+1}{2}-\frac{\ell-1}{2})\cdot
\left(\substack{
\frac{\ell'+1}{2}-1
\\
\frac{\frac{\ell'+1}{2}+\frac{\ell''+1}{2}-\frac{\ell-1}{2}}{2}
}\right)\!
\left(\substack{
\frac{\frac{\ell'+1}{2}+\frac{\ell-1}{2}+\frac{\ell''+1}{2}}{2}-2
\\
\frac{\ell'+1}{2}-1
}\right)\; ,\\
\hspace{6cm}\text{if $\frac{\ell'+1}{2}+\frac{\ell-1}{2}+\frac{\ell''+1}{2}$
even\; ,}
\end{cases}
\end{multline*}
that is,
\begin{multline*}
\mathfrak{T}\left(\alpha,\beta; \tfrac{\ell-1}{2}, \tfrac{\ell'+1}{2}, \tfrac{\ell''+1}{2}\right)\\=
\begin{cases}
\frac{\ell'+\ell''-\ell+3}{2}
\left(\substack{
\frac{\ell'-1}{2}
\\
\frac{\ell'+\ell''-\ell+1}{4}
}\right)\!
\left(\substack{
\frac{\ell+\ell'+\ell''-5}{4}
\\
\frac{\ell'-1}{2}
}\right)\; ,\ \ \ \  \text{if $\frac{\ell+\ell'+\ell''+1}{2}$
odd\; ,}\\
\quad\\
\left(\substack{
\frac{\ell'-1}{2}
\\
\frac{\ell'+\ell''-\ell-1}{4}
}\right)\!
\left(\substack{
\frac{\ell+\ell'+\ell''-3}{4}
\\
\frac{\ell'-1}{2}
}\right)
+\frac{\ell'+\ell''-\ell+3}{2}
\left(\substack{
\frac{\ell'-1}{2}
\\
\frac{\ell'+\ell''-\ell+3}{4}
}\right)\!
\left(\substack{
\frac{\ell+\ell'+\ell''-7}{4}
\\
\frac{\ell'-1}{2}
}\right)\; ,
\\ \hspace{6.4cm}\text{if $\frac{\ell+\ell'+\ell''+1}{2}$
even\; .}
\end{cases}
\end{multline*}
We see that the number of pairs~$(A,B)$ of sets~$A$ and~$B$ in the family~(\ref{eq:1}), with the properties given in~(\ref{eq:3}),
can be calculated by means of~(\ref{eq:11}).

Assertions~(vi), (vii), (viii) and~(ix) are analogues of assertions~(ii), (iii), (iv) and~(v), respectively.
\end{proof}

\section*{Appendix: Enumeration of Smirnov words over three-letter and four-letter alphabets
in the framework of~\cite{Prodinger}}
\label{appendixpage}

The words without consecutive equal letters (\!{\em waves}, {\em normal words}), called the {\em Smirnov words\/} after the work~\cite{SSZ}, are investigated, applied and enumerated, e.g., in~\cite{A,AvMS,Carlitz-72}\cite[Examples~7.45, 8.14 and 8.16]{Do},\cite{DGG,ERR,EW,Farmer},\cite[Examples~III.24 and~IV.10]{FS},\cite{FHP, Gafni,Gessel-Thesis},\cite[\S{}2.4 and Exercise~3.5.1]{GJ},\cite[pp.~164-166]{Honsberger},\cite{KCh,KT,KG,WVLR,Lea,Li-Thesis,Lientz,MacFie},\cite[Examples~2.2.10 and 13.3.5]{PW-Book},\cite[\S{}4.8]{PW},\cite{Prodinger,RSh,RL,SZ,ShW-Advances,ShW-Edizioni,Sundaram,Taylor,Crux}. The importance of these words can easily be explained~\cite[p.~205]{FS}:
\begin{quote}
{\small
Start from a Smirnov word and substitute for any letter~$a_j$ that appears in it an arbitrary nonempty
sequence of letters~$a_j$. When this operation is done at all places of a Smirnov word,
it gives rise to an unconstrained word. Conversely, any word can be associated to a unique
Smirnov word by collapsing into single letters maximal groups of contiguous equal letters.
}
\end{quote}

\subsection*{Ternary Smirnov words}

Let $(\theta,\alpha,\beta)$ be a three-letter alphabet, and let `$\mathrm{u}$',`$\mathrm{v}$' and `$\mathrm{w}$' be formal variables which mark the letters $\theta$, $\alpha$ and $\beta$, respectively. Let $\boldsymbol{\mathtt{e}}_{\centerdot}(\!\cdots\!)$ denote elementary symmetric polynomials.
The ordinary trivariate generating function
of the set of ternary Smirnov words~is
\begin{equation}
\label{eq:17}
\begin{split}
\frac{1}{1-\left(\frac{\mathrm{u}}{1+\mathrm{u}}+\frac{\mathrm{v}}{1+\mathrm{v}}+\frac{\mathrm{w}}{1+\mathrm{w}}\right)}
&=\frac{(1 + \mathrm{u}) (1 + \mathrm{v}) (1 + \mathrm{w})}{1-(\mathrm{u}\mathrm{v}+\mathrm{u}\mathrm{w}+\mathrm{v}\mathrm{w}+2\mathrm{u}\mathrm{v}\mathrm{w})}
\\&=
\frac{(1 + \mathrm{u}) (1 + \mathrm{v}) (1 + \mathrm{w})}{1-\boldsymbol{\mathtt{e}}_2(\mathrm{u},\mathrm{v},\mathrm{w})-2\boldsymbol{\mathtt{e}}_3(\mathrm{u},\mathrm{v},\mathrm{w})}
\; ;
\end{split}
\end{equation}
see~\cite[Example~III.24]{FS} and~\cite[\S{}2.4.16]{GJ} on the general multivariate generating function of the Smirnov words.

In the theoretical framework of~\cite{Prodinger}, let us consider the system of generating functions
\begin{equation*}
\begin{cases}
f_{\theta}(\mathrm{u},\mathrm{v},\mathrm{w}) = \mathrm{u} + \mathrm{u}f_{\alpha}(\mathrm{u},\mathrm{v},\mathrm{w})
+ \mathrm{u}f_{\beta}(\mathrm{u},\mathrm{v},\mathrm{w})\; ,\\
f_{\alpha}(\mathrm{u},\mathrm{v},\mathrm{w}) =\phantom{\mathrm{u} + }\; \mathrm{v}f_{\theta}(\mathrm{u},\mathrm{v},\mathrm{w})
+ \mathrm{v}f_{\beta}(\mathrm{u},\mathrm{v},\mathrm{w})\; ,\\
f_{\beta}(\mathrm{u},\mathrm{v},\mathrm{w}) = \phantom{\mathrm{u} + }\; \mathrm{w}f_{\theta}(\mathrm{u},\mathrm{v},\mathrm{w})
+ \mathrm{w}f_{\alpha}(\mathrm{u},\mathrm{v},\mathrm{w})
\end{cases}
\end{equation*}
(where again the formal variable `$\mathrm{u}$' marks the letters $\theta$, the variable `$\mathrm{v}$'
marks the letters~$\alpha$, and the variable `$\mathrm{w}$' marks the letters $\beta$)
rewritten, for short, as
\begin{equation}
\label{eq:201}
\begin{cases}
f_{\theta} = \mathrm{u} + \mathrm{u}f_{\alpha} + \mathrm{u}f_{\beta}\; ,\\
f_{\alpha} = \phantom{\mathrm{u} + }\; \mathrm{v}f_{\theta} + \mathrm{v}f_{\beta}\; ,\\
f_{\beta} = \phantom{\mathrm{u} + }\; \mathrm{w}f_{\theta} + \mathrm{w}f_{\alpha}\; .
\end{cases}
\end{equation}
For a letter~$\mathfrak{s}\in(\theta,\alpha,\beta)$, the generating function~$f_{\mathfrak{s}}:=f_{\mathfrak{s}}(\mathrm{u},\mathrm{v},\mathrm{w})$ is meant to count the ternary
Smirnov words starting with the letter~$\theta$ and ending with the letter~$\mathfrak{s}$.

The solutions to the system~(\ref{eq:201}) are
\begin{align*}
f_{\theta}=\frac{\mathrm{u}(1-\mathrm{v}\mathrm{w})}{1-(\mathrm{u}\mathrm{v}+\mathrm{u}\mathrm{w}+\mathrm{v}\mathrm{w}
+2\mathrm{u}\mathrm{v}\mathrm{w})}
&=\frac{\mathrm{u}(1-\boldsymbol{\mathtt{e}}_2(\mathrm{v},\mathrm{w}))}{1-\boldsymbol{\mathtt{e}}_2(\mathrm{u},\mathrm{v},\mathrm{w})
-2\boldsymbol{\mathtt{e}}_3(\mathrm{u},\mathrm{v},\mathrm{w})}\; ,\\
f_{\alpha}=\frac{\mathrm{u}\mathrm{v}(1+\mathrm{w})}{1-(\mathrm{u}\mathrm{v}+\mathrm{u}\mathrm{w}+\mathrm{v}\mathrm{w}
+2\mathrm{u}\mathrm{v}\mathrm{w})}
&=\frac{\mathrm{u}\mathrm{v}(1+\mathrm{w})}{1-\boldsymbol{\mathtt{e}}_2(\mathrm{u},\mathrm{v},\mathrm{w})
-2\boldsymbol{\mathtt{e}}_3(\mathrm{u},\mathrm{v},\mathrm{w})}
\\
\intertext{and}
f_{\beta}&
=\frac{\mathrm{u}\mathrm{w}(1+\mathrm{v})}{1-\boldsymbol{\mathtt{e}}_2(\mathrm{u},\mathrm{v},\mathrm{w})
-2\boldsymbol{\mathtt{e}}_3(\mathrm{u},\mathrm{v},\mathrm{w})}
\; ,
\end{align*}
cf.~(\ref{eq:17}).
Thus, we have\footnote{Requests of the form\newline
{\tt series of} {\em rational\_generating\_function}\newline
and\newline
{\tt series of} {\em rational\_generating\_function} {\tt wrt} {\it formal\_variable}\newline
were made to the online \href{http://www.wolframalpha.com/}{\em W\!olfram{$|$}\!Alpha} computational knowledge engine.}
\begin{align}
\label{eq:18}
f_{\theta}&=\sum_{k=1}^{\infty}\left(\frac{\mathrm{v}+\mathrm{w}+2\mathrm{v}\mathrm{w}}{1-\mathrm{v}\mathrm{w}}\right)^{k-1}\mathrm{u}^k
\\
\label{eq:205}
&=
\underbrace{
(1-\boldsymbol{\mathtt{e}}_2(\mathrm{v},\mathrm{w}))\sum_{k=1}^{\infty}\frac{(\boldsymbol{\mathtt{e}}_1(\mathrm{v},\mathrm{w})
+2\boldsymbol{\mathtt{e}}_2(\mathrm{v},\mathrm{w}))^{k-1}}{(1-\boldsymbol{\mathtt{e}}_2(\mathrm{v},\mathrm{w}))^k}\mathrm{u}^k
}_{
\text{cf.~(\ref{eq:207})}
}
\; ,
\\
\label{eq:206}
f_{\alpha}
&=
\underbrace{
\mathrm{v}(1+\mathrm{w})\sum_{k=1}^{\infty}\frac{(\boldsymbol{\mathtt{e}}_1(\mathrm{v},\mathrm{w})
+2\boldsymbol{\mathtt{e}}_2(\mathrm{v},\mathrm{w}))^{k-1}}{(1-\boldsymbol{\mathtt{e}}_2(\mathrm{v},\mathrm{w}))^k}\mathrm{u}^k
}_{
\text{cf.~(\ref{eq:208})}
}
\\
\intertext{and}
\label{eq:20}
f_{\beta}&=
\mathrm{w}(1+\mathrm{v})\sum_{k=1}^{\infty}\frac{(\mathrm{v}+\mathrm{w}+2\mathrm{v}\mathrm{w})^{k-1}}{(1-\mathrm{v}\mathrm{w})^k}\mathrm{u}^k
\\
\nonumber
&=
\mathrm{w}(1+\mathrm{v})\sum_{k=1}^{\infty}\frac{(\boldsymbol{\mathtt{e}}_1(\mathrm{v},\mathrm{w})
+2\boldsymbol{\mathtt{e}}_2(\mathrm{v},\mathrm{w}))^{k-1}}{(1-\boldsymbol{\mathtt{e}}_2(\mathrm{v},\mathrm{w}))^k}\mathrm{u}^k
\; .
\end{align}

\begin{remark}
\label{th:2}
In the framework of\/~\cite{Prodinger}, the numbers $\mathfrak{T}(\theta,\mathfrak{s};k,i,j)$ of distinct ternary Smirnov words that start with the letter~$\theta$, end with a letter~\mbox{$\mathfrak{s}\in\{\theta,
\beta\}$}, and contain $k$ letters $\theta$, $i$ letters~$\alpha$, and $j$
letters~$\beta$, can be read off from the power series representations of the generating functions~$f_{\mathfrak{s}}$, given in~{\rm(\ref{eq:18})} and~{\rm(\ref{eq:20})}, as the coefficients~$[\mathrm{u}^k\mathrm{v}^i\mathrm{w}^j]f_{\mathfrak{s}}$ of $\mathrm{u}^k\mathrm{v}^i\mathrm{w}^j$:
\begin{itemize}
\item[\rm(i)]
\begin{equation}
\label{eq:15}
\mathfrak{T}(\theta,\theta;k,i,j)=\begin{cases}
\left(\substack{
k-1
\\
\frac{k+i-j-1}{2}
}\right)\!
\left(\substack{
\frac{k+i+j-3}{2}
\\
k-2
}\right)\; , & \text{if $k+i+j$
odd\; ,}\\
\quad\\
(k+i-j)\cdot
\left(\substack{
k-1
\\
\frac{k+i-j}{2}
}\right)\!
\left(\substack{
\frac{k+i+j}{2}-2
\\
k-2
}\right)\; , & \text{if $k+i+j$
even\; .}
\end{cases}
\end{equation}

\item[\rm(ii)]
\begin{equation}
\label{eq:16}
\mathfrak{T}(\theta,\beta;k,i,j)
=\begin{cases}
(k+j-i)\cdot
\left(\substack{
k-1
\\
\frac{k+j-i-1}{2}
}\right)\!
\left(\substack{
\frac{k+i+j-3}{2}
\\
k-1
}\right)\; ,\ \  \text{if $k+i+j$
odd\; ,}\\
\quad\\
\left(\substack{
k-1
\\
\frac{k+j-i}{2}-1
}\right)\!
\left(\substack{
\frac{k+i+j}{2}-1
\\
k-1
}\right)
+(k+j-i)\cdot
\left(\substack{
k-1
\\
\frac{k+j-i}{2}
}\right)\!
\left(\substack{
\frac{k+i+j}{2}-2
\\
k-1
}\right)\; ,\\
\hspace{6cm}\text{if $k+i+j$
even\; .}
\end{cases}
\end{equation}
\end{itemize}
\end{remark}

\subsection*{Smirnov words over a four-letter alphabet}

The ordinary quadrivariate generating function of the set of Smirnov words over the alphabet~$(\theta,\alpha,\beta,\gamma)$
with its four letters marked by the formal variables `$\mathrm{u}$', `$\mathrm{v}$', `$\mathrm{w}$' and `$\mathrm{x}$', is
\begin{gather*}
\frac{1}{1-\left(\frac{\mathrm{u}}{1+\mathrm{u}}+\frac{\mathrm{v}}{1+\mathrm{v}}+\frac{\mathrm{w}}{1+\mathrm{w}}+\frac{\mathrm{x}}{1+\mathrm{x}}\right)}
\\=\frac{(1 + \mathrm{u}) (1 + \mathrm{v}) (1 + \mathrm{w}) (1 + \mathrm{x})}
{1-(
\mathrm{u}\mathrm{v}+\mathrm{u}\mathrm{w}+\mathrm{u}\mathrm{x}+\mathrm{v}\mathrm{w}
+\mathrm{v}\mathrm{x}+\mathrm{w}\mathrm{x}
+2(\mathrm{u}\mathrm{v}\mathrm{w}+\mathrm{u}\mathrm{v}\mathrm{x}+\mathrm{u}\mathrm{w}\mathrm{x}+\mathrm{v}\mathrm{w}\mathrm{x})+3
\mathrm{u}\mathrm{v}\mathrm{w}\mathrm{x})}
\\
=
\frac{(1 + \mathrm{u}) (1 + \mathrm{v}) (1 + \mathrm{w}) (1 + \mathrm{x})}
{1-\boldsymbol{\mathtt{e}}_2(\mathrm{u},\mathrm{v},\mathrm{w},\mathrm{x})
-2\boldsymbol{\mathtt{e}}_3(\mathrm{u},\mathrm{v},\mathrm{w},\mathrm{x})-3\boldsymbol{\mathtt{e}}_4(\mathrm{u},\mathrm{v},\mathrm{w},\mathrm{x})}
\; .
\end{gather*}

For letters $\mathfrak{s}\in(\theta,\alpha,\beta,\gamma)$, the generating functions $g_{\mathfrak{s}}:=g_{\mathfrak{s}}(\mathrm{u},\mathrm{v},\mathrm{w},\mathrm{x})$ defined by the system
\begin{equation}
\label{eq:202}
\begin{cases}
g_{\theta} = \mathrm{u}+\mathrm{u}\cdot(g_{\alpha} + g_{\beta}+ g_{\gamma})\; ,\\
g_{\alpha} = \phantom{\mathrm{u}+}\ \mathrm{v}\cdot(g_{\theta} + g_{\beta}+ g_{\gamma})\; ,\\
g_{\beta} = \phantom{\mathrm{u}+}\ \mathrm{w}\cdot(g_{\theta} + g_{\alpha}+ g_{\gamma})\; ,\\
g_{\gamma} = \phantom{\mathrm{u}+}\ \ \! \mathrm{x}\cdot(g_{\theta} + g_{\alpha}+ g_{\beta})\; ,
\end{cases}
\end{equation}
are intended for enumerating the Smirnov words starting with the letter $\theta$ and ending with the letter $\mathfrak{s}$.

The solutions to the system~(\ref{eq:202}) are
\begin{align}
\nonumber
g_{\theta}&=\frac{\mathrm{u}(1-(\mathrm{v}\mathrm{w}+\mathrm{v}\mathrm{x}+\mathrm{w}\mathrm{x}))}{1-(
\mathrm{u}\mathrm{v}+\mathrm{u}\mathrm{w}+\mathrm{u}\mathrm{x}+\mathrm{v}\mathrm{w}
+\mathrm{v}\mathrm{x}+\mathrm{w}\mathrm{x}
+2(\mathrm{u}\mathrm{v}\mathrm{w}+\mathrm{u}\mathrm{v}\mathrm{x}+\mathrm{u}\mathrm{w}\mathrm{x}+\mathrm{v}\mathrm{w}\mathrm{x})+3
\mathrm{u}\mathrm{v}\mathrm{w}\mathrm{x})}
\\
\nonumber
&=\frac{\mathrm{u}(1-\boldsymbol{\mathtt{e}}_2(\mathrm{v},\mathrm{w},\mathrm{x}))}{1-
\boldsymbol{\mathtt{e}}_2(\mathrm{u},\mathrm{v},\mathrm{w},\mathrm{x})
-2\boldsymbol{\mathtt{e}}_3(\mathrm{u},\mathrm{v},\mathrm{w},\mathrm{x})
-3\boldsymbol{\mathtt{e}}_4(\mathrm{u},\mathrm{v},\mathrm{w},\mathrm{x})}
\\
\label{eq:203}
&=(1-(\mathrm{v}\mathrm{w}+\mathrm{v}\mathrm{x}+\mathrm{w}\mathrm{x}))\sum_{k=1}^{\infty}\frac{(\mathrm{v}+\mathrm{w}+\mathrm{x}
+2(\mathrm{v}\mathrm{w}+\mathrm{v}\mathrm{x}+\mathrm{w}\mathrm{x})
+3\mathrm{v}\mathrm{w}\mathrm{x})^{k-1}}
{(1-(\mathrm{v}\mathrm{w}+\mathrm{v}\mathrm{x}+\mathrm{w}\mathrm{x}+2\mathrm{v}\mathrm{w}\mathrm{x}))^k}\mathrm{u}^k
\\
\label{eq:207}
&=
\underbrace{
(1-\boldsymbol{\mathtt{e}}_2(\mathrm{v},\mathrm{w},\mathrm{x}))\sum_{k=1}^{\infty}\frac{(\boldsymbol{\mathtt{e}}_1(\mathrm{v},\mathrm{w},\mathrm{x})
+2\boldsymbol{\mathtt{e}}_2(\mathrm{v},\mathrm{w},\mathrm{x})
+3\boldsymbol{\mathtt{e}}_3(\mathrm{v},\mathrm{w},\mathrm{x}))^{k-1}}
{(1-\boldsymbol{\mathtt{e}}_2(\mathrm{v},\mathrm{w},\mathrm{x})-2\boldsymbol{\mathtt{e}}_3(\mathrm{v},\mathrm{w},\mathrm{x}))^k}\mathrm{u}^k
}_{
\text{cf.~(\ref{eq:205})}
}
\; ,
\\
\nonumber
g_{\alpha}&=\frac{\mathrm{u}\mathrm{v}(1+\mathrm{w})(1+\mathrm{x})}{1-(
\mathrm{u}\mathrm{v}+\mathrm{u}\mathrm{w}+\mathrm{u}\mathrm{x}+\mathrm{v}\mathrm{w}
+\mathrm{v}\mathrm{x}+\mathrm{w}\mathrm{x}
+2(\mathrm{u}\mathrm{v}\mathrm{w}+\mathrm{u}\mathrm{v}\mathrm{x}+\mathrm{u}\mathrm{w}\mathrm{x}+\mathrm{v}\mathrm{w}\mathrm{x})+3
\mathrm{u}\mathrm{v}\mathrm{w}\mathrm{x})}
\\
\nonumber
&=\frac{\mathrm{u}\mathrm{v}(1+\mathrm{w})(1+\mathrm{x})}{1-
\boldsymbol{\mathtt{e}}_2(\mathrm{u},\mathrm{v},\mathrm{w},\mathrm{x})
-2\boldsymbol{\mathtt{e}}_3(\mathrm{u},\mathrm{v},\mathrm{w},\mathrm{x})
-3\boldsymbol{\mathtt{e}}_4(\mathrm{u},\mathrm{v},\mathrm{w},\mathrm{x})}
\\
\label{eq:204}
&=\mathrm{v}(1+\mathrm{w})(1+\mathrm{x})\sum_{k=1}^{\infty}\frac{(\mathrm{v}+\mathrm{w}+\mathrm{x}
+2(\mathrm{v}\mathrm{w}+\mathrm{v}\mathrm{x}+\mathrm{w}\mathrm{x})
+3\mathrm{v}\mathrm{w}\mathrm{x})^{k-1}}
{(1-(\mathrm{v}\mathrm{w}+\mathrm{v}\mathrm{x}+\mathrm{w}\mathrm{x}+2\mathrm{v}\mathrm{w}\mathrm{x}))^k}\mathrm{u}^k
\\
\label{eq:208}
&=
\underbrace{
\mathrm{v}(1+\mathrm{w})(1+\mathrm{x})\sum_{k=1}^{\infty}\frac{(\boldsymbol{\mathtt{e}}_1(\mathrm{v},\mathrm{w},\mathrm{x})
+2\boldsymbol{\mathtt{e}}_2(\mathrm{v},\mathrm{w},\mathrm{x})
+3\boldsymbol{\mathtt{e}}_3(\mathrm{v},\mathrm{w},\mathrm{x}))^{k-1}}
{(1-\boldsymbol{\mathtt{e}}_2(\mathrm{v},\mathrm{w},\mathrm{x})-2\boldsymbol{\mathtt{e}}_3(\mathrm{v},\mathrm{w},\mathrm{x}))^k}\mathrm{u}^k
}_{
\text{cf.~(\ref{eq:206})}
}
\; ,
\\
\nonumber
g_{\beta}&=\frac{\mathrm{u}\mathrm{w}(1+\mathrm{v})(1+\mathrm{x})}{1-
\boldsymbol{\mathtt{e}}_2(\mathrm{u},\mathrm{v},\mathrm{w},\mathrm{x})
-2\boldsymbol{\mathtt{e}}_3(\mathrm{u},\mathrm{v},\mathrm{w},\mathrm{x})
-3\boldsymbol{\mathtt{e}}_4(\mathrm{u},\mathrm{v},\mathrm{w},\mathrm{x})}
\\
\nonumber
&=\mathrm{w}(1+\mathrm{v})(1+\mathrm{x})\sum_{k=1}^{\infty}\frac{(\boldsymbol{\mathtt{e}}_1(\mathrm{v},\mathrm{w},\mathrm{x})
+2\boldsymbol{\mathtt{e}}_2(\mathrm{v},\mathrm{w},\mathrm{x})
+3\boldsymbol{\mathtt{e}}_3(\mathrm{v},\mathrm{w},\mathrm{x}))^{k-1}}
{(1-\boldsymbol{\mathtt{e}}_2(\mathrm{v},\mathrm{w},\mathrm{x})-2\boldsymbol{\mathtt{e}}_3(\mathrm{v},\mathrm{w},\mathrm{x}))^k}\mathrm{u}^k
\intertext{and}
\nonumber
g_{\gamma}&=\frac{\mathrm{u}\mathrm{x}(1+\mathrm{v})(1+\mathrm{w})}{1-
\boldsymbol{\mathtt{e}}_2(\mathrm{u},\mathrm{v},\mathrm{w},\mathrm{x})
-2\boldsymbol{\mathtt{e}}_3(\mathrm{u},\mathrm{v},\mathrm{w},\mathrm{x})
-3\boldsymbol{\mathtt{e}}_4(\mathrm{u},\mathrm{v},\mathrm{w},\mathrm{x})}
\\
\nonumber
&=\mathrm{x}(1+\mathrm{v})(1+\mathrm{w})\sum_{k=1}^{\infty}\frac{(\boldsymbol{\mathtt{e}}_1(\mathrm{v},\mathrm{w},\mathrm{x})
+2\boldsymbol{\mathtt{e}}_2(\mathrm{v},\mathrm{w},\mathrm{x})
+3\boldsymbol{\mathtt{e}}_3(\mathrm{v},\mathrm{w},\mathrm{x}))^{k-1}}
{(1-\boldsymbol{\mathtt{e}}_2(\mathrm{v},\mathrm{w},\mathrm{x})-2\boldsymbol{\mathtt{e}}_3(\mathrm{v},\mathrm{w},\mathrm{x}))^k}\mathrm{u}^k
\; .
\end{align}

\begin{remark}
The numbers $\mathfrak{F}(\theta,\mathfrak{s};k,i,j,h)$ of distinct Smirnov words, over the four-letter alphabet $(\theta,\alpha,\beta,\gamma)$ and with the Parikh vector $(k,i,j,h)$,
that start with the letter~$\theta$ and end with a letter~\mbox{$\mathfrak{s}\in\{\theta,
\alpha\}$}, can be read off, in one way or another, from the power series representations of the generating functions~$g_{\mathfrak{s}}$, given in~{\rm(\ref{eq:203})} and~{\rm(\ref{eq:204})}, as the coefficients~$[\mathrm{u}^k\mathrm{v}^i\mathrm{w}^j\mathrm{x}^h]g_{\mathfrak{s}}$
of~$\mathrm{u}^k\mathrm{v}^i\mathrm{w}^j\mathrm{x}^h$. For example, we have:
\begin{itemize}
\item[\rm(i)]
\begin{multline}
\mathfrak{F}(\theta,\theta;k,i,j,h)=
\\
\sum_{
\substack{
0\;\leq\; p\; \leq\; k-1\; ,\\
0\;\leq\; r\; \leq\; \lfloor\frac{1}{2}(i+j+h-k+1)\rfloor
}
}
\ \
\sum_{
\substack{
p\;\leq\; s\; \leq\; k-1\; ,\\
r\;\leq\; t\; \leq\; \lfloor\frac{1}{2}(i+j+h-k+1)\rfloor
}
}
\tbinom{t+k-1}{k-s-1,\ \ p,\ \ s-p,\ \ r,\ \ t-r}
\\
\times
2^{-i-j-h+3k-2s-2r+4t-3}\cdot 3^{i+j+h-2k+s+r-3t+2}
\\
\times
\Biggm(
\tbinom{k-s-1}{-i-j-h+3k-2s-r+3t-3}
\tbinom{-i-j-h+3k-2s+3t-3}{-i+k+p-s+t-1}
\tbinom{-j-h+2k-2p+2t-2}{-j+k-p+t-1}
\\
-
\tfrac{4}{9}\cdot
\tbinom{k-s-1}{-i-j-h+3k-2s-r+3t-1}
\tbinom{-i-j-h+3k-2s+3t}{-i+k+p-s+t}
\tbinom{-j-h+2k-2p+2t}{-j+k-p+t}
\Biggm)
\; .
\end{multline}

\item[\rm(ii)]
\begin{multline}
\mathfrak{F}(\theta,\alpha;k,i,j,h)=
\\
\sum_{
\substack{
0\;\leq\; p\; \leq\; k-1\; ,\\
0\;\leq\; r\; \leq\; \lfloor\frac{1}{2}(i+j+h-k)\rfloor
}
}
\ \
\sum_{
\substack{
p\;\leq\; s\; \leq\; k-1\; ,\\
r\;\leq\; t\; \leq\; \lfloor\frac{1}{2}(i+j+h-k)\rfloor
}
}
\tbinom{t+k-1}{k-s-1,\ \ p,\ \ s-p,\ \ r,\ \ t-r}
\\
\times
2^{-i-j-h+3k-2s-2r+4t-2}\cdot 3^{i+j+h-2k+s+r-3t+1}
\\
\times
\Biggm(
\tbinom{k-s-1}{-i-j-h+3k-2s-r+3t-2}
\tbinom{-i-j-h+3k-2s+3t-2}{-i+k+p-s+t}
\tbinom{-j-h+2k-2p+2t-2}{-j+k-p+t-1}
\\
+
\tfrac{2}{3}\;\cdot
\Bigm(
\tbinom{k-s-1}{-i-j-h+3k-2s-r+3t-1}
\tbinom{-i-j-h+3k-2s+3t-1}{-i+k+p-s+t}
\\
+\tfrac{2}{3}\!\cdot\!
\tbinom{k-s-1}{-i-j-h+3k-2s-r+3t}
\tbinom{-i-j-h+3k-2s+3t}{-i+k+p-s+t}
\Bigm)\cdot\;\tbinom{-j-h+2k-2p+2t}{-j+k-p+t}
\Biggm)\; .
\end{multline}

\end{itemize}
\end{remark}

\vspace{5mm}
\end{document}